\documentclass[11pt]{article}

\usepackage[margin=1in]{geometry}  
\usepackage{graphicx}              
\usepackage{amsmath}               
\usepackage{amsfonts}              
\usepackage{amsthm}                
\usepackage{amssymb}

\newtheorem{theorem}{Theorem}
\newtheorem{lemma}{Lemma}
\newtheorem{proposition}{Proposition}
\theoremstyle{definition}
\newtheorem{definition}{Definition}

\theoremstyle{remark}
\newtheorem{corollary}{Corollary}
\newtheorem{remark}{Remark}

\begin{document}

\nocite{*}

\title{Gradient flow on Finsler manifolds}

\author{N. Shojaee, M. M Rezaii\thanks{The corresponding author} \\ 
Department of  Mathematics and Computure science\\
Amirkabir University (Polytechnic of Tehran)\\
Tehran, Iran}

\maketitle
\begin{abstract}
studying various functionals and associated gradient flows are known problems in differential geometry.
The perpose of this article is to provide a general overview of curvature functionals in Finsler geometry and use their  information for 
introducing gradient flow on Finsler manifolds. For aiming this purpose, at first we prove space of Finslerian metrics is a Riemannian manifold, then we give some decompositions of tangent space of this manifold and finally we introduce gradient flow by useing Akbar Zadeh curvature functional.
\end{abstract}
\section{Introduction}
Nonlinear heat flows first appeared in Riemannian geometry in 1964, when Eells and Sampson introduced the harmonic map heat flow as the gradient flow of the energy functional $E(u)=\int_{M}|\nabla u|^2dV$ \cite{ES}. They used this flow as a tool to deform given maps $u:M\rightarrow N$ between two manifolds into extremal maps which are critical points in the sense of calculuse of varitions for $E(u)$.\\
A fundamental problem in differential geometry is to find canonical metrics on Riemannian manifolds, i.e. metrics which are highly symmetrical, for example metrics with constant curvature in some sense. Hamilton used the idea of evolving an object to such an ideal state by a nonlinear heat flow for the first time and invented the Ricci flow in 1981 \cite{HA}. He proved that a Riemannian metric of strictly positive Ricci curvature on a compact 3-manifold can be deformed into a metric of positive constant curvature, using this idea G. Huisken \cite{HU}, C. Margerin \cite{MA} and S. Nishikawa \cite{NI} proved that on a compact n-manifold, a Riemannian metric can be deformed into a metric of constant curvature, if it is sufficiently close to a metric of positive constant curvature \cite{V}.\\
The stationary metrics under the Ricci flow are Ricci flat metrics which are also the critical points of the Einstein -Hilbert functional ${\cal E}(g)=\int_MRdV$ but the Ricci flow is not exactly the gradient flow of functional ${\cal E}(g)=\int_MRdV$, it is just a part of Einstein -Hilbert functional's gradient flow, $\partial_tg_{ij}=-R_{ij}+\frac{R}{n}g_{ij}$.  If this functional is restricted to the class of conformal metrics then it has a strictly parabolic gradient flow which is called Yamabe flow. Hamilton proved that there is not any functional such that its gradient flow is Ricci flow. Perelman improved Einstien-Hilbert functional and introduced ${\cal F}(g_{ij},f)=\int_M(|\nabla f|^2+R)e^{-f}dV$,this functional has a system of including two gradient flows which one of them is Ricci flow \cite{M}.\\
Since gradient flows has important role in global analysis on manifolds and in different branches of application sciences like image processing, biological problems and ..., it is an important branch of studying. So it might be good to extend this topic on Finslerian manifolds. For the first attempt in this topic, we can mention the concept of Ricci flow by Bao \cite{B1} and as a bit more serios studying, Ohta and Sturm introduced heat flows on Finsler manifolds \cite{OS}.\\
Since the classification of Finslerian manifolds with constant curvature is incomplete and furthermore, there are different kinds of curvature in Finsler geometry, we can not assure are gradient flows benefit tools for studying on Finsler manifolds like Riemannian ones?\\
The main object of this paper is to introduced gradient flow of curvature functionals in Finsler geometry. For aiming this purpose, we start with studying space of Finsler metrics ${\cal M}_F$ and then we produced some decompositions for tangent space of ${\cal M}_F$ and in the last part of our paper we give an exact definition of variations of Finslerian metrics and using the calculuse  of variations for deriving some gradient flows.
\section{Preliminaries}
Let $(M,g)$ be a connected, compact Finsler manifold. It means that there is a function $F$ on tangent bundle $TM$ with the following conditions:
\begin{itemize}
\item $F$ is a smooth function on the entire slit tangent bundle $TM_o$.
\item $F$ is a positive homogenuouse function on second variable, $y$.
\item The matrix $(g_{ij})$, $g_{ij}(x,y)=\frac{1}{2}\frac{\partial^2F^2}{\partial y^i\partial y^j}$ is nondegenerate.
\end{itemize}
The geodesics of a Finsler structure $F$ are characterized locally by $\frac{d^2x^i}{dt^2}+2G^i(x,\frac{dx}{dt})=0$ where  
$G^i=\frac{1}{4}g^{ih}(\frac{\partial^2 F^2}{\partial y^h\partial x^j}y^j-\frac{\partial F^2}{\partial x^h})$
and called geodesic spray coefficients.
Set $G^i_j=\frac{\partial G^i}{\partial y^j}$ which are the coefficients of nonlinear connection on $TM$. By means of this nonlinear connection, tangent space  of $TM_o$ splits in two horizontal and vertical subspaces, which is spanned by $\{\frac{\delta}{\delta x^i},\frac{\partial}{\partial y^i}\}$, where $\frac{\delta}{\delta x^i}:=\frac{\partial}{\partial x^i}-G_i^j\frac{\partial}{\partial y^j}$ that are called Berwald bases and their dual bases are denoted by $\{dx^i,\delta y^i\}$, where $\delta y^i:=dy^i+G_j^idx^j$. Furthermore this nonlinear connection can be used to define a linear connection which is called Berwald connection and its one forms defined locally by $\pi^i_j=G^i_{jk}dx^k$ where $G^i_{jk}=\frac{\partial G^i_j}{\partial y^k}$.
The one forms of cartan connection are defined by $\nabla\frac{\partial}{\partial x^i}=\omega_i^j\frac{\partial}{\partial x^j}$, where for $\Gamma_{jk}^i=\frac{1}{2}g^{im}(\frac{\partial g_{mj}}{\partial x^k}+\frac{\partial g_{mk}}{\partial x^j}-\frac{\partial g_{kj}}{\partial x^m})-(C^i_{js}G^s_k+C^i_{ks}G^s_j-C_{kjs}G^{si})$ and $C^i_{jk}=\frac{1}{2}g^{im}(\frac{\partial g_{mj}}{\partial y^k}+\frac{\partial g_{mk}}{\partial y^j}-\frac{\partial g_{kj}}{\partial y^m})$, $\omega_j^i=\Gamma^i_{jk}dx^k+C^i_{jk}\delta y^k$.Now, by using definition of curvature tensor on Riemannian vector bundle $\pi^*TM$, hh-curvature of Cartan and Berwald connections are related by \cite{A2},
$$R^{~i}_{j~kl}= H^{~i}_{j~kl}+C^{i}_{jr}R^{~r}_{o~kl}+\nabla_k\nabla_oC^{i}_{jk}-\nabla_k\nabla_o C^{i}_{jl}+\nabla_oC^{i}_{lr}\nabla_oC^{r}_{ji}-\nabla_oC^i_{kr}\nabla_oC^{r}_{jl}$$
Indicatrix bundle of a Finsler structure is defined by $SM:=\mathop{\cup}\limits_{x\in M} S_xM$, where $ S_x=\{y\in T_xM|F(x,y)=1\}$, according to this definition, $S_xM$ is the hypersurface
in $T_xM$. Indicatrix bundle, $SM$ is always orientable and since we assume $M$ is compact, $SM$ is compact, too. These two properties help to define volume form and global inner 
product on $SM$. The volume element of the indicatrix bundle is denoted by $(2n-1)$-form $\eta$ \cite{A2},
$$\eta:=\frac{(-1)^N}{(n-1)}\phi,\quad \phi=\omega\wedge(d\omega)^{(n-1)},\quad N=\frac{n(n-1)}{2}$$
where $\omega$ is Hilbert form. On tensor spaces on $SM$, we have the canonical scalar product (point-wise) $<.|.>$ and on their sections, the global scalar product $(.|.)=\int_{SM}<.|.>\eta$.
Furthermore, Akbarzadeh introduced the codifferential operator on the differentiable one forms defined on $SM$ in \cite{A2},
\begin{equation}
\delta a=-(\nabla^ja_j-a_j\nabla_0T^j)
\end{equation}
where $a$ is a horizontal 1-form on $SM$. And,
\begin{equation}
\delta b=-F(\dot{\nabla}_jb^j+b_jT^j)=-Fg^{ij}\partial_jb_i
\end{equation}
where $b$ is a vertical 1-form on $SM$, $\nabla$ and $\dot{\nabla}$ are horizontal and vertical coeffiecients of Finslerian(Cartan) connection $\tilde{\nabla}=\nabla+\dot{\nabla}$, respectively.
The Ricci tensor is introduced from different ways in Finsler geometry, in this paper we consider Akbarzade definition ${\tilde H}_{ij}=\frac{\partial^2H_{rs}}{\partial y^i\partial y^j}y^ry^s$ where $H_{ij}=y^kH_{ikjs}y^s$ which is defined by hh-curvature. Another curvature which is defined by hh-curvature is Ricci-directional curvature, $H(u,u)=g^{ik}H_{ijkl}u^ju^l$. This curvature is related to second type scalar curvature $\tilde{H}=g^{ij}\tilde{H}_{ij}$, in the critical points of Akbarzadeh curvature functionals.
\section{Space of Finsler metrics}
Endow compact manifold $M$, with a fixed Finslerian structure $F$ that making it a Finslerian manifold $(M,g)$. Since all coefficients of a Finslerian metric are zero homognuous, it is sufficient to consider them on $SM$. So without loss of generality, we can write $g\in S^2_+(\pi^*_sT^*M)$ where $\pi_s:SM\rightarrow M$ and $S^2_+$ denotes space of positive definite two forms. We note that Finsler metrics are special case of the GL-metrics on $TM$. In other words, a GL-metric $g_{ij}(x,y)$ is reducible to a Finsler metric if and only if the vertical coefficients of the cartan connection $C_{ijk}$, satisfied $C_{ijk}y^j=0$ \cite{BR},
but according to the Euler theorem and definition of cartan tensor, this property is always true when coefficients, $C_{ijk}$ are symmetric in all three indices.\\
\begin{proposition}
The space of all Finslerian metrics on a compact manifold $M$ is a Riemannian manifold \cite{}.
\end{proposition}
\begin{proof}
The relation $C_{ijk}y^j=0$ is simplifying to the linear differential equation $y^k\frac{\partial g_{ij}}{\partial y^k}=0$  on a Finsler manifold $M$. It is clear that this PDE is solvable.
Let $D$ be a domain of $SM$ and we have a set of differential equations for a collection of functions $$(g_{11},\dots,g_{1n},\dots,g_{n1},\dots,g_{nn}):D\rightarrow {\mathbb R}^{n\times n}$$
Then solutions of the above differential equation are sections of ${\mathbb R}^{n\times n}$ fibered over $D$. The collection of these sections is an infinite dimensional vector space and so it is an infinite dimensional manifold which is represented by ${\cal M}_F$.
For every $g\in {\cal M}_F$, the tangent space of this manifold is the space of all symmetric two tensors that are zero homogenuouse and symmetric in all three indices i.e, 
$T_g {\cal M}_F=\{h\in S^2(\pi^*_sT^*M)| y^j\frac{\partial h_{ij}}{\partial y^k}=0\}$.
we assume that $h$'s are squar integrable of order s and define global inner product on ${\cal M}_F$ by $$(a,b):=\sum_{s\geq 0}\int_{SM}<\tilde{\nabla}^sa,\tilde{\nabla}^sb>\eta$$ where $a,b\in T_g {\cal M}_F$ and $\tilde{\nabla}$ is cartan connection. This inner product just depends on $x$. Therefore, the pair $({\cal M}_F, (.|.))$ is a Riemannian manifold.
\end{proof}
\begin{definition}
The set of all Finslerian metrics on a compact manifold $M$ is defined as an answer set of linear partial differential equation $y^iC_{ijk}=0$ and represented by ${\cal M}_F$.
\end{definition} 
\section{Different decompositions of tangent space of ${\cal M}_F$}
Let $X=X^i\frac{\partial}{\partial x^i}$ be a section of $\Gamma(\pi^*TM)$. Define a uniqe associated horizontal vector field by $\hat X=X^i\frac{\delta}{\delta x^i}$. Consider canonical linear mapping $\varrho:T_zTM\rightarrow \pi^*T_xM$ that is $\varrho_z(\frac{\delta}{\delta x^i})=\frac{\partial}{\partial x^i}|_x$ and 
$\varrho_z(\frac{\partial}{\partial y^i})=0$ in local coordinates. Suppose $\hat X,\hat Y$ and $\hat Z$ are sections of $\Gamma(TTM)$ so the Lie derivative of Finslerian metric $g$ is defined by $(L_{\hat X}g)(\varrho\hat Y,\varrho\hat Z)$ by using Lie derivative and torsion definitions and properties of cartan connection, we obtian:
\begin{align*}
L_{\hat X}g(\varrho\hat Y,\varrho\hat Z)&=L_{\hat X}g(Y,Z)\\
&=g(symm(\nabla X)\hat Y,Z)+g(Y,symm(\dot{\nabla} X)\hat Z)\\
&+2g(T(X,\dot{Z}),Y)+g(T(\dot{X},Z),Y)+g(T(\dot X,Y),Z)
\end{align*}
where $g(symm(\nabla X)\hat Y,Z):=g(\nabla_{H\hat Y} X,Z)+g(Y,\nabla_{H\hat Z} X)$ and is defined for vertical connection similar to this.\\
Now, suppose vector field $\hat X$ is a complete lift of a vector field $X$ on $M$ by replacing this vector field in Lie derivative equation, and using $y^m\frac{\partial X^i}{\partial x^m}=y^m\frac{\delta X^i}{\delta x^m}$ and 
\begin{align*}
\nabla_{(X^iG^l_i+y^m\frac{\partial X^l}{\partial x^m})\frac{\partial}{\partial x^l}}\frac{\partial}{\partial x^k}&=(X^iG^l_i+y^m\frac{\partial X^l}{\partial x^m})C^m_{lk}\frac{\partial}{\partial x^m}\\
&=(y^m\frac{\delta X^l}{\delta x^m}+y^mX^iF^l_{im})C^m_{lk}\frac{\partial}{\partial x^m}\\
&=y^m\nabla_mX^lC^m_{lk}\frac{\partial}{\partial x^m}
\end{align*}
we deduced that 
\begin{align}
L_{\hat X}g(Y,Z)=\nabla_iX_j+\nabla_jX_i+2y^m\nabla_mX^iC_{kij}
\end{align}
By means of the global inner product, we define the adjoint of this operator.
\begin{lemma}
Let $(M,g)$ be a compact Finslerian manifold and $h$ be an arbitrary symmetric two form of $S^2\pi^*T^*M$, the adjoint of Lie derivative of $h$ in local coordinates is
\begin{align}
\delta h=-(\nabla^ih_{ik}-h_{kj}\nabla_0T^j+\dot{C}_{kij}h^{ij}+C_{kij}\nabla_oh^{ij})\label{div}
\end{align}
\end{lemma}
\begin{proof}
\begin{align*}
\int_{SM}\frac{1}{2}(L_{\hat X}g,h)\eta&=\frac{1}{2}\int_{SM}(\nabla_iX_j+\nabla_jX_i+2y^m\nabla_mX^kC_{ijk})h^{ij}\eta\\
&=\int_{SM}\nabla_iX_jh^{ij}\eta+\int_{SM}y^m\nabla_mX^kC_{kij}h^{ij}\eta\\
&=\int_{SM}(h_{ik}\nabla_0T^i-\nabla^ih_{ij}-(\nabla_0C_{ijk})h^{ij}-C_{ijk}\nabla_0h^{ij})X^k\eta\\
&=-\int_{SM}(\nabla^ih_{ik}-h_{ik}\nabla_0T^i+\dot{C}_{kij}h^{ij}+C_{kij}\nabla_oh^{ij})X^k\eta\\
&=\int_{SM}(X,\delta h)\eta
\end{align*}
\end{proof}
\begin{definition}
Divergence of symmetric two forms in $S^2\pi^*T^*M$ is adjoint of $L_{\hat X}g$ with respect to global inner product and calculated by \ref{div} and represented by $\delta$.
\end{definition}
\begin{theorem}
The Berger-Ebin decomposition for $T_g{\cal M}_F$ is $T_g{\cal M}_F=\{h|h=L_{\hat X}g\}\oplus S^T$ where $S^T:=\{h|\delta_gh=0\}$
\end{theorem}
\begin{proof}
We define the differential operator $\tau_g$ for every $g\in {\cal M}_F$ by $\tau_g h:=-\pi_{s*}\delta_g h$ which is an operator from $T_g{\cal M}_F$ to $\Gamma TM$. Its adjoint is denoted by $\tau^*$ and define by $L_{\hat X}g$ where $\hat X$ is a complete lift of $X$.
For an arbitrary one form $t$ on $SM$, the symbole of $\tau$ is defined by $\sigma_t(\tau):=\varrho^*t\otimes X^h_{\sharp}+ X^h_{\sharp}\otimes \varrho^*t$ and it is injective so the Berger-Ebin decomposition of $T_g{\cal M}_F$ is $Im\tau\oplus ker\tau^*$.
\end{proof}
\begin{definition}
Section $X$ of the tangent bundle $TM$ is a Finslerian killing vector field if its complete lift $\hat{X}$ on $TM$ is a killing vector field for Finslerian metric $g$, that is $L_{\hat X}g=0$. 
\end{definition}
A point-wise conformal deformation of a Finslerian metric $g$, is $\tilde{g}(x,y)=f(x)g(x,y)$ where $f$ is a smooth positive function on $M$,\cite{Kn}.Since there is a one to one corresponding between space of positive functions and space of exponential functions
by $f\rightarrow e^f$, we can write $\tilde{g}=e^fg$. Therefore, let ${\cal P}$ be the product group of positive function on $M$ that acts on ${\cal M}_F$ by function $A$ as follows:
\begin{align*}
A:{\cal P}\times{\cal M}_F\rightarrow {\cal M}_F\\
A(f,g):=f.g
\end{align*}
This action is free and smooth. The orbit of this action at $g\in{\cal M}_F$ is defined by$A_g=\{fg|f\in{\cal P}\}$ which is a submanifold of ${\cal M}_F$ \cite{S}. Tangent space of this manifold at $g$ is ${\cal F}g=\{h=kg|k\in C^{\infty}(M)\}$ where is a subspace of $S^2\pi^*T^*M$ at each point $g\in {\cal M}_F$. 
Orthogonal subspace of ${\cal F}g$ with respect to the global inner product is $\{h\in S^2\pi^*_sT^*M|\int_{SM}kgh\eta=0\}=\{h\in S^2\pi^*_sT^*M|tr(h)=0\}$. On the other hand, from the variation of the volume forms \cite{A2}, we have $tr(h)=0$ is equivalent of being constant volume on $SM$. So the orthogonal space of ${\cal F}g$ is the space of two forms which preserve volume $SM$ through metric variations. Thus there is a point wise decomposition like 
\begin{align}
S^2\pi^*_sT^*M={\cal F}g\oplus S^T\label{1}
\end{align}
Let $D$ be the group of infinitesimal diffeomorphism on $M$ and ${\cal P}$ be a one parameter group of positive function on $M$. Put ${\cal C}={\cal D}\ltimes{\cal P}$ which is a semi-direct group with the following action:
\begin{align*}
(\eta_1,f_1).(\eta_2,f_2)=(\eta_1 o \eta_2,f_2.(f_1 o \eta_2))
\end{align*}
This group acts on ${\cal M}_F$ by function $\tilde{A}$ as follows:
\begin{align*}
{\tilde A}:{\cal C}\times{\cal M}_F\rightarrow {\cal M}_F\\
\tilde{A}((\eta, f),g)=f.(\tilde{\eta}^*g)
\end{align*}
where $\tilde{\eta}$ is the natural extention of $\eta$ on $TM$ which is defined by $\tilde{\eta}_t:(x^i,y^i)\rightarrow (x^i+tv^i, y^i+ty^m\frac{\partial v^i}{\partial x^m})$ such that $v^i$s are components of vector field $V$ on $M$ which is inducing infinitesimal point transformation $\eta_t$ .
 It is clear that $\hat{V}:=\frac{d}{dt}|_{t=0}\tilde{\eta}_t$ is a complete lift of the vector field $V$ on $TM$. The orbit of ${\tilde A}$ passing through $g\in\cal{M}_F$ is
\begin{align*}
{\tilde A}_g:{\cal F}\rightarrow {\cal M}_F\\
{\tilde A}_g(\eta,g)=f.(\tilde{\eta}^*g)
\end{align*}
which is a submanifold of ${\cal M}_F$. So we dfine $\tau_g:=d{\tilde A}_g|_{(e,1)}$ as follows:
\begin{align*}
\tau_g:\Gamma(TM)\times{\cal F}\rightarrow T_g{\cal M}_F\\
\tau_g(X,f)=L_{\hat X}g+kg
\end{align*}
The adjoint of $\tau_g$ is denoted by $\tau^*_g$ and defined by:
\begin{align*}
\tau^*_g:T_g{\cal M}_F\rightarrow\Gamma(TM)\times{\cal F}\\
h\rightarrow(\pi_{s*}(\sharp div h),\int_{SM}tr(h)\eta)
\end{align*}
The kernel of this map is $S^{TT}=\{h\in T_g{\cal M}_F|\sharp div h=0, \int_{SM}tr(h)\eta=0\}$, and since the symbol of the map $\tau_g$ i.e. $\sigma_t(\tau_g)(X,f)=fg+t\bigotimes \pi_s^* X_{\sharp}+\pi_s^* X_{\sharp}\bigotimes t$ where $t$ is an arbitrary one form on $SM$ is injective so the Berger-Ebin decomposition is
\begin{align*}
S^2\pi_s^*T^*M=S^{TT}\oplus Im\tau_g
\end{align*}
By corresponding this decomposition with point-wise decomposition \ref{1} ,we have
\begin{align*}
S^2\pi_s^*T^*M={\cal F}g\oplus S^{TT}\oplus (S^T\cap Im\tau_g)
\end{align*}
The last term of the right hand side of the above equation indicates that every two form $h=L_{\hat X}g+fg$ preserve volume of $SM$ that is $tr(h)=0$ so $h$ must be in the form $h=L_{\hat X}g-\frac{2}{n}div(\hat{X})g$. All of the above discusion can be summarized in the following theorem.
\begin{theorem}
The Berger-Ebin decomposition for $T_g{\cal M}_F$ according to conformal deformation of metrics is $T_g{\cal M}_F={\cal F}g\oplus S^{TT}\oplus (S^T\cap Im\tau_g)$.
\end{theorem}
\section{Curvature functionals on $\cal{M}$}
\begin{definition}
A variation of a Finslerian metric, $g_o$ is a one-parameter family of this metric i.e, $\{g_{t}\}_{t\in I}$ where $g_t=g_o+th$ ,$g_o\in{\cal M}_F$
and $h\in T_g {\cal M}_F$.
\end{definition}
According to the above definition, a variation of Finslerian metric is a curve in infinite dimensional manifold ${\cal M}_F$,
with tangent vector field, $h:=g'_t=\frac{\partial g_t}{\partial t}$.
When a Finslerian metric has been deformed, geometric structures, like nonlinear coefficients, curvatures, volume forms and indicatrix 
will be changed
. Variations of these objects are 
calculated in \cite{A2},
\begin{align}
\eta'&=(g^{ij}-\frac{n}{2}u^iu^j)h_{ij}\eta\label{eq0}\\
G'^i_j&=\frac{1}{2}(\nabla_kh^i_o+\nabla_oh^i_k-\nabla^ih_{ok})-2C^i_{ks}G'^s \label{eq1}\\
R'^{i}_{~jkl}&=\nabla_k\Lambda^i_{~jl}-\nabla_l\Lambda^i_{~jk}+P^i_{~jlr}\Lambda^r_{~ok}-P^i_{~jkr}\Lambda^r_{ol}+C'^i_{~jr}R^r_{~okl}
\end{align}
where
$$\Lambda^i_{~jk}=\Gamma'^i_{~jk}+C^i_{~jr}\Gamma'^r_{ok}$$ and 
\begin{align*}
\Gamma'^i_{~jk}&=\frac{1}{2}g^{im}(\nabla_k h_{mj}+\nabla_jh_{mk}-\nabla_mh_{jk})\\
&-(C^i_{~js}G'^s_k+C^i_{~ks}G'^s_j-C_{kjs}G'^s_mg^{im})
\end{align*}
Akbazade by using differen kinds of curvatures introduced some functionals in\cite{A1,A2} like $I(g_t)=\int_{SM}{\hat H_t}\eta_t$ where ${\hat H}={\tilde H}-c(x)H(u,u)$ which is a more general case of other ones and by means of the above variations find their critical points and through this way, he defined generalized Einstien metrics(GEM) on Finsler manifolds as critical points of these functionals based on the assumption that indicatrix has constant volume. Since all thier functionals caused to GEM so we just choose the general one,$I(g_t)$ and works with it. According to Akbarzade's calculation, the Euler-Lagrange equation of these functionals is $-\tilde{H}_{ij}+c(x)g_{ij}=0$, so thier associated gradient flow is $\frac{\partial g_{ij}}{\partial t}=-\tilde{H}_{ij}+H(u,u)g_{ij}=-\nabla I(g_t)$. But there is a question here, is there any answer for this equation? For answering this question, consider the linearization of this equation. Since $\tilde{H_{ij}}'y^iy^j=\frac{1}{2}\frac{\partial}{\partial y^i}\frac{\partial}{\partial y^j}H'_{kr}y^ky^r$ and $H'_{ks}y^ky^s=2\nabla_rG'^r-\nabla_0G'^r_r+2\nabla_0T_rG'^r$\cite{A2} and use \ref{eq1}, we deduced that the first term of the gradient flow is of order 4 in term of $h$. For eliminating this problem, we contract both of the equation with $u^i$ and $u^j$ so we have $\partial_t \log F_t=0$ which is degenerate. So we consider one of the two terms by contracting, we obtain $\partial_t \log F_t=-H(u,u)$. This equation is introduced by Bao \cite{B1} and called Ricci flow. This equation like in Riemannian case is not a gradient flow.\\
\begin{proposition}
The functional $I(g_t)$ is a Riemannian functional.
\end{proposition}
\begin{proof}
Suppose $\varphi$ is a diffeomorphism on $TM$ and $g$ a Finslerian metric on $M$, so $\varphi^*g\in {\cal M}_F$. With straight calculation for one parameter family of diffeomorphism obtained $\tilde{H}_{\varphi_t^*g_0}=\varphi_t^*\tilde{H}_{g_0}$, $H(u,u)_{\varphi_t^*g_0}=\varphi_t^*H(u,u)_{g_0}$ and $\eta_{\varphi_t^* g}=\varphi_t^*\eta_{g_0}$ by replacing these equations in $I(g_t)$, we obtain:
\begin{align*}
I(\varphi_t^*(g_0))&=\int_{\tilde{SM}}\hat{H}_{\varphi_t^*g_0}\eta_{\varphi_t^*g_0}=\int_{\varphi_t^*(SM)}\varphi_t^*\hat{H}_{g_0}\varphi_t^*\eta_{g_0}=\int_{\varphi_t^*(SM)}\varphi_t^*(\hat{H}_{g_0}\eta_{g_0})\\
&=\varphi_t^*(I(g_0))
\end{align*}
Taking derivative from both sides of above equation:
$$\frac{d}{dt}|_{t=0}I(\varphi_t^*g_0)=I'(g_0)L_Xg_0=\frac{d}{dt}|_{t=0}\varphi_t^*(I(g_0))=X.I(g_0)=0$$
So $I(\varphi_t^*(g_0))=I(g_0)$and it means $I(g_t)$ is a Riemannian functional. 
\end{proof} 
\begin{remark}
Curvature functional $I(g_t)$ is invariant under one parameter group of diffeomorphisms so we can consider the curvature functional on the space $\frac{{\cal M}_F}{{\cal D}}$. specially $I(g_t)$ is invariant under Finsler killing vector fields.
\end{remark}
According to the Berger-Ebin decomposition, there are two kinds of conformal deformation, a point-wise conformal deformation and infinitesimal conformal deformation that the first one takes part in ${\cal F}g$ and the other takes part in $S^{TT}$. For our purpose, we restrict Akbar Zadeh's introducing functionals on ${\cal F}g$.
\begin{lemma}
The variation of volume form with respect to the point-wise conformal variation is $\eta'=\frac{1}{2}tr_g(h)\eta$.
\end{lemma}
\begin{proof}
The point-wise conformal variation of a metric $g$ is ${\tilde g}_{ij}=e^{2f(t,x)}g_{ij}$ so $h_{ij}=\varrho(t,x)g_{ij}$ where $\varrho(t,x)=f'(t,x)e^{f(t,x)}=\frac{1}{n}tr_g(h)$ by replacing this equation in \ref{eq0} we obtain $\eta'=\frac{1}{2n}tr_g(h)\eta$.
\end{proof}
\begin{theorem}
Let $M$ be a closed and connected Finslerian manifold with dim$\geq 3$. A metric $g$ is critical for $I(g_t)$ under all pointwise conformal variations if and only if the Finslerian manifold is Ricci directional flat. 
\end{theorem}
\begin{proof}
The derivative of functional $I(g)$ in usual direction is 
\begin{align}
({\tilde H}_{jk}-\lambda H(u,u)u_ju_k-(n\tau-\phi)u_ju_k-\frac{1}{2}\hat{H}g_{jk})h^{jk}=0\label{eq2}
\end{align}
We product two side of equaion by $u^iu^j$ and obtain
\begin{align}
{\tilde H}(u,u)-\lambda H(u,u)-(n\tau-\phi)-\frac{{\hat H}}{2}=0\label{eq3}
\end{align}
Since $t^{jk}=\frac{tr_g(t)}{n}g^{jk}$, so we product both side of \ref{eq2} by $g^{jk}$, 
\begin{align}
{\tilde  H}-\lambda H(u,u)-(n\tau-\phi)-\frac{n}{2}{\hat H}=0\label{eq4}
\end{align}
Now from equations \ref{eq3} and \ref{eq4}, we deduced that 
\begin{align}
\frac{n-1}{2}{\hat H}=-\frac{n-1}{2}a+{\tilde H}-{\tilde H}(u,u)\label{eq5}
\end{align}
and 
\begin{align}
\lambda H(u,u)+(n\tau-\phi)=\frac{n}{n-1}{\tilde H}(u,u)-\frac{1}{n-1}{\tilde H}\label{eq6}
\end{align}
Replacing two last equations in \ref{eq2}. By simplifying the equation and product it to $u^iu^j$, we have
$$(n-2){\tilde H}(u,u)=0$$
so $H(u,u)=0$
\end{proof}
This functional is not invariant under scaling. For eliminating this problem, we use a normal factor $\psi=\psi(t)$, and put $\tilde{g}=\psi(t)g(t)$ such that $\int_{SM}\tilde{\eta}=1$. So we deduced that $\eta=\psi^{\frac{-n}{2}}\tilde{\eta}$ and by replacing it in volume formula, we have $\psi=(V(t))^{\frac{-2}{n}}$. Now, we rewrite the functional $I(g_t)$ with respect to this normal factor;
\begin{align*}
\tilde{I}(g)&=I(\tilde{g}_t)=\int_{SM}(H(\tilde{g})-\lambda H(u,u)(\tilde{g}))\tilde{\eta}\\
&=\int_{SM}\psi^{-1}(H(g)-\lambda H(u,u)(g))\psi^{\frac{n}{2}}\eta\\
&=\psi^{\frac{n-2}{2}}I(g)\\
&=(V(t))^{\frac{2-n}{n}}I(g)
\end{align*}
\begin{theorem}
Let $M$ be a closed and connected Finslerian manifold with dim$\geq 3$. A metric $g$ is critical for $\tilde{I}(g_t)$ under all pointwise conformal variations if and only if the Finslerian manifold is of constant Ricci-directional curvature. 
\end{theorem}
\begin{proof}
Taking derivative from both sides of equation $\tilde{I}(g_t)=(V(t))^{\frac{2-n}{n}}I(g)$ and calculate it at $t=0$:
\begin{align*}
\tilde{I}'(g_t)|_{t=0}&=\frac{2-n}{n}V(t)'|_{t=0}(V(0))^{\frac{2-n}{n}-1}I(g_0)+v(0)^{\frac{2-n}{n}}I'(g_t)|_{t=0}\\
&=V(0)^{\frac{2-n}{n}}\{\frac{2-n}{2n}\frac{I(g_0)}{V(0)}\int_{SM}tr(h)\eta+\int_{SM}A_{ij}h^{ij}\eta\}
\end{align*}
Put $Ave:=\frac{I(g_0)}{V(0)}$ which is a constant value. with restricted to point-wise conformal deformation, we have:
\begin{align}
0&=\tilde{I}'(g_t)|_{t=0}\nonumber\\
&=V(0)^{\frac{2-n}{n}}\int_{SM}(\frac{2-n}{n}Ave+A_{ij}g^{ij})\frac{tr_g(h)}{n}\eta
\end{align}
Now we try to simplify paranteces equation;
\begin{align}
0=\frac{2-n}{n}Ave+A_{ij}g^{ij}=\frac{2-n}{n}Ave-\tilde{H}+\lambda H(u,u)+(n\tau-\varphi)+\frac{n}{2}\hat{H}\label{eq7}
\end{align}
By replacing \ref{eq5} and \ref{eq6} in \ref{eq7} we deduced
\begin{align*}
H(u,u)=-\frac{(n-2)(n-1/2)}{8n}Ave
\end{align*}
\end{proof}
\begin{corollary}
Under assumption of above theorem, second type scalar curvature is constant, too.
\end{corollary}
\begin{proof}
In the stationary points of cuurvature functional $I(g_t)$ based on constant indicatrix volume, we have $nH(u,u)=\tilde{H}$\cite{A2}. So $\tilde{H}$ is constant, too.
\end{proof}
\begin{definition}
The normalized gradient flow of functional $I(g_t)$ with restricted to the point-wise conformal deformation is $\frac{\partial}{\partial t}g_{ij}=-(H(u,u)-c)g_{ij}$ where $c$ is consant value.It is clear that unnormalized gradient flow is defined by $\frac{\partial}{\partial t}g_{ij}=-H(u,u)g_{ij}$.
\end{definition}
\begin{corollary}
Unnormalized gradient flow $\frac{\partial}{\partial t}g_{ij}=-H(u,u)g_{ij}$ is a weakly parabolic equation.
\end{corollary}
\begin{proof}
The derivative of functional $I(g)$ with above assumptions is $I'(g_t)|_{t=0}=\int_{SM}H(u,u)tr_g(t)\eta=\int_{SM}H(u,u)g_{jk}t^{jk}\eta=0$ so its Euler-lagrange equation is $H(u,u)g_{ij}=0$ and its associated gradient flow is $\frac{\partial}{\partial t}g_{ij}=-H(u,u)g_{ij}$.
The linearization of this equation is
\begin{align*}
{\tilde H}(u,u)=F^{-2}\xi_i\xi_jy^jt^i_0-\frac{F^{-2}}{4}\xi_i\xi_jg^{ij}t_{00}+\text{lower order terms.}
\end{align*}
Put $\xi_1=1$ and $\xi_j=0$ for all $j\neq 1$. To evaluate the symbol of this equation, we take an orthonormal frame $(e_i)$ at $x\in M$ such that $u^n=\frac{y^n}{F}=1$ and $u^\alpha=0$ for all $\alpha\neq n$, it is clear that
\begin{align*}
(\sigma D(E)(g_{jk})(\xi)({\tilde g}))_{jk}=-\frac{F^{-2}}{4}t_{00}\quad j=k=1\quad \text{and for all other case is zero}
\end{align*}
So it is a weakly parabolic equation.
\end{proof}

Since scalar form of both equations are same so it seems that scalar form is not such a good form for studying flows. On the other hand, you saw there is not any difference between choosing second part of definition GEM or first part of it for introducing Ricci flow but with last proposition, we can say that the tensor form of Ricci flow is $\frac{\partial}{\partial t}g_{ij}=-\tilde{H}_{ij}$ since we use the second term for introducing a new version of flow in Finsler geometry.


\bibliographystyle{plain}

\bibliography{paper}

\begin{thebibliography}{99}
\bibitem{A1}
H. Akbar-Zadeh, Initiation to Global Finslerian Geometry, North-
Holland Mathematical Library, vol; 68, 2006.
\bibitem{A2}
H. Akbar-Zadeh, Generalized Einstein manifolds, Journal of Geometry and Physics, Vol;17, 342-380 (1995).
\bibitem{A3}
H. Akbar-Zadeh, Sur les espaces de Finsler la courbures sectionnelles constantes,
Acad. Roy. Belg. Bull. Cl. Sci. 74 (1988), 281-322.
\bibitem{B1}
D. Bao, On two curvature-driven problems in Riemann-Finsler geometry, Adv.
stud. pure Math. 48 (2007), 19-71.
\bibitem{BC}
D. Bao, S.S. Chern, Z. Shen, Riemann-Finsler geometry, Springer-Verlag,
2000.
\bibitem{BCS}
D. Bao, S.S. Chern and Z. Shen, An Introduction to Riemann-Finsler Geometry,
Graduate Texts in Mathematics, vol; 200, Springer , 2000.
\bibitem{BR}
D. Bao and C. Robles, Ricci and flag curvatures in Finsler geometry, MSRI pub, Vol; 50, (2004).
\bibitem{BE2}
M. Berger and D. Ebin, Some decompositions of the space of symmetric tensors on a Riemannian
manifold, J. Differ. Geom., 3, No. 3, 379–392 (1969).
\bibitem{B}
A. L. Besse, Einstein Manifolds, Springer-Verlag, Berlin (1987).
\bibitem{Bl}
D. E. Blair, Contact manifolds in Riemannian geometry, Lect. Notes Math., 509 (1976).
\bibitem{BM}
I. Bucataru, R. Miron, Finsler-Lagrange geometry, Applications to dynamical systems, CEEX ET 3174/2005-2007 and CEEX M III 12595/2007.
\bibitem{CH}
B. Chow, The Yamabe flow on locally conformally flat manifolds with positive Ricci curvature comm, Pure Appl. Math. Vol;XIV, 1003-1014 (1992).
\bibitem{CK}
B. Chow and D. Knopf, The Ricci Flow, An Introduction, Mathematical Surveys and
Monographs 110, Amer. Math. Soc., Providence, RI, 2004.
\bibitem{DE}
Sh. Deng, Homogeneous Finsler Space, Springer Monographs in Mathematics, 59-73 (2012).
\bibitem{E1}
D. Ebin, “The manifold of Riemannian metrics,” Proc. Symp. Pure Math., 15, 11–40 (1970).
\bibitem{ES}
J. Eells J. H. Sampson, Harmonic mappings of Riemannian manifolds, American Journal of Mathematics, Vol. 86, No. 1 (Jan., 1964), 109-160.
\bibitem{F1}
E. Fischer, E. Marsden, The manifold of conformally equivalent metrics, Can. J. Math., Vol;XXIX, no 1, 193-209 (1977). 
\bibitem{Kn}
M. S. Knebelman, Conformal geometry of generalized metric spaces, Proc. nat. Acad. Sci. USA 15, 376-379(1929)
\bibitem{KR}
W. Kuhnel and H. Rademacher.: Einstein spaces with a conformal group, Result. Math 56, 421-444, (2009). 
\bibitem{L1}
S. Lang, Differential Manifolds, Addison Wesley, Reading, Mass., 1972.
\bibitem{OS}
Sh. Ohta and K. T. Sturm, Heat flow on Finsler manifolds, arXive.org., 2012.
\bibitem{PB}
P. Joharinad, B. Bidabad, Conformal vector fields on Finsler spaces, Differential
Geometry and its Applications, Volume 31, Issue 1, February
2013, 33-40
\bibitem{HA}
R. S. Hamilton, Three-manifolds with positive Ricci curvature, J. Differential
Geom. 17 (1982), no. 2, 255-306.
\bibitem{Ha}
M. Hashiguchi, On conformal transformations of Finsler metrics, J. Math. Kyoto Univ, 16-1(1979), 25-50.
\bibitem{HU}
G. Huisken, Ricci deformation of the metric on a Riemannian manifold, J. Differential Geom. 21 (1985),
47-62.
\bibitem{MA}
C. Margerin, Pointwise pinched manifolds are space forms, Proc. Sympos. Pure Math. 44 (1986), 307-328.
\bibitem{M}
R. Muller,Differential Harnack Inequalities and the Ricci Flow, EMS Series of Lectures in Mathematics, 2006.
\bibitem{MT}
V.S., Matveev, M., Troyanov, The Binet–Legender ellipsoid in Finsler geometry. Math. DG
arxiv: 1104–1647.v1, preprint. Accessed 2012.
\bibitem{NI}
S. Nishikawa, Deformation of Riemannian metrics and manifolds with bounded curvature ratios, Proc.
Sympos. Pure Math., Vol. 44 (1986), 343-352.
\bibitem{V}
Handbook of differential geometry,
\bibitem{S}
N. K. Smolentsev, Spaces of Riemannian metrics, Journal of Mathematical Sciences, Vol; 142, No. 5, 2007.
\bibitem{VI}
J. Viaclovsky, Topics in Riemannian Geometry, Lecture notes, Fall (2008).

\end{thebibliography}

\end{document}